\documentclass{amsart}
\usepackage{amsfonts}
\usepackage{amsmath,amsthm}
\usepackage{amsfonts,amssymb}
\usepackage{multirow}
\usepackage{booktabs}
\usepackage{enumerate}

\newtheorem{thm}{Theorem}[section]

\newtheorem{lem}[thm]{Lemma}

\newtheorem{rem}[thm]{Remark}

\numberwithin{equation}{section}

\newcommand{\al}{\alpha}

\def \b{\beta}

\def\vz{\varepsilon}
\def\oz{\omega}
\def\lz{\lambda}
\def\Lz{\Lambda}

\def\az{\alpha}

\def\sz{\sigma}

\def\({\Bigl(}
\def \){ \Bigr)}

\def\sub{\substack}

 \def\RR{{\mathbb R}}

\def\sz{\sigma}

\def\x{{\bf x}}
\def\y{{\bf y}}

\def\kk{{\bf k}}

\def\va{\varepsilon}

\begin{document}
\def\RR{\mathbb{R}}
\def\Exp{\text{Exp}}
\def\FF{\mathcal{F}_\al}

\title[] {On the power of standard information for tractability for
  $L_\infty$ approximation of periodic functions in the worst case setting}

\author{Jiaxin Geng} \address{ School of Mathematical Sciences, Capital Normal
University, Beijing 100048,
 China.}
 \email{gengjiaxin1208@163.com}

\author{Heping Wang} \address{ School of Mathematical Sciences, Capital Normal
University,
Beijing 100048,
 China.}
\email{wanghp@cnu.edu.cn}


\begin{abstract} We study multivariate approximation of periodic function  in the worst case setting  with
the error measured in the  $L_\infty$ norm. We consider
algorithms that use standard information $\Lz^{\rm std}$
consisting of function values or general linear information
$\Lz^{\rm all}$ consisting of arbitrary continuous linear
functionals. We investigate the equivalences of various notions
of algebraic
 and exponential tractability  for
$\Lz^{\rm std}$  and $\Lz^{\rm all}$  under the absolute or
normalized error criterion, and show that  the power of $\Lz^{\rm
std}$  is the same as the one  of $\Lz^{\rm all}$  for some
notions of algebraic  and exponential tractability. Our result can be
applied to weighted Korobov spaces and Korobov spaces with
exponential weight. This gives a special solution to Open problem
145 as posed by Novak and Wo\'zniakowski (2012) \cite{NW3}.

\end{abstract}

\maketitle
\input amssym.def

\section{Introduction}

\

We study  multivariate approximation
$I=\{{I}_{\infty,d}\}_{d\in\Bbb N}$, where
$${{I}}_{\infty,d}: H(K_d)\to L_\infty(D_d)\ \ {\rm with}\ \  {I}_{\infty,d}\,
(f)=f ,$$ is the compact embedding operator,  $H(K_d)$ is  a
separable reproducing kernel  Hilbert function space on $D_d$ with
kernel $K_d$, $D_d\subset \Bbb R^d$, and the dimension $d$ is
large or even huge. We also investigate approximation problem
${\rm APP}=\{{{\rm APP}}_{\infty,d}\}_{d\in\Bbb N}$, where
$${{\rm APP}}_{\infty,d}: H^\omega(\Bbb T^d)\to L_\infty(\Bbb T^d)\ \ {\rm with}\ \  {\rm APP}_{\infty,d}\,
(f)=f ,$$where $\Bbb T^d=[0,1]^d$ is $d$-dimensional torus and
$H^\omega(\Bbb T^d)$ is smoothness space of
multivariate periodic functions, whose definition is given in
Subsection 2.2.

 We
consider algorithms that use finitely many information
evaluations. Here information evaluation means continuous linear
functional on $H(K_d)$ (general linear information) or function
value at some point (standard information). We use $\Lz^{\rm all}$
and $\Lz^{\rm std}$ to denote the  class of all continuous linear
functionals and the  class of all function values, respectively.

 For a given error threshold
$\vz\in(0,1)$, the information complexity $n(\vz, d)$ is defined
to be the minimal number of information evaluations for which the
approximation error of some algorithm is at most $\vz$.
Tractability is aimed at studying how the information complexity
$n(\vz,d)$ depends on $\vz$ and $d$. There are two kinds of
tractability based on polynomial convergence and exponential
convergence. The algebraic tractability (ALG-tractability)
describes how the information complexity $n(\vz, d)$ behaves as a
function of $d$ and $\vz^{-1}$,  while the exponential
tractability (EXP-tractability)
 does as one of $d$ and $(1+\ln\vz^{-1})$.
 In
recent years   the study of algebraic  and
 exponential tractability has attracted much interest, and a great number of
interesting results
 have been obtained
  (see
\cite{CW, DKPW,  GW, IKPW, KW, LX2, NW1, NW2, NW3, PP,  S1, SiW,
W, X1, X2} and the references therein).

This paper is devoted to discussing  the power of standard
information for $L_\infty$ approximation problem
 in the worst case setting. The class $\Lz^{\rm std}$ is much smaller and much more
practical, and is much more difficult to analyze than the class
$\Lz^{\rm all}$. Hence, it is very important to study the power of
$\Lz^{\rm std}$ compared to $\Lz^{\rm all}$. There are many papers
devoted to this field. For example, for the randomized setting,
see \cite{CDL, CM,   K19, KWW2, LW2, NW3, WW07};  for the average
case setting, see \cite{HWW, LZ, LW1, NW3,  X5};  for the worst
case setting see \cite{HNV,  KSUW,  KS, KU, KU2, KWW,  NSU, NW3,
NW16, NW17, PU, WW1}.

However, most of the above papers are about $L_2$ approximation
and there are only three papers \cite{KWW, KWW2, PU} devoted to
$L_\infty$  approximation. In this paper we  consider tractability
of  the $L_\infty$ approximation problem defined over reproducing
kernel
 Hilbert space with reproducing kernel being  multivariate periodic smooth function in the worst case setting.  We obtain the
equivalences of various notions of algebraic and exponential
tractability for $\Lz^{\rm all}$ and $\Lz^{\rm std}$ for the
normalized or absolute error criterion without any condition.
    This gives a special solution to Open
problem 145 in \cite{NW3}.

This paper is organized as follows.   In Subsections 2.1  we
introduce  multivariate approximation in reproducing kernel
Hilbert spaces.  Approximation of multivariate periodic functions
is given in Subsection 2.2. The various notions of algebraic and
exponential tractability are given in Subsection 2.3.
   Our main results Theorems
  2.1-2.4
 are stated  in Subsection 2.4. In  Section 3, we give the proofs of Theorems 2.1-2.3. After that, in Section 4  we show Theorem 2.4, i.e.,  the
equivalences of various  notions of tractability for the absolute
 or normalized error criterion in the worst case
setting. Finally, in Section 5 we give applications of Theorem 2.4
to  weighted Korobov spaces and Korobov spaces with exponential
weight.

\section{Preliminaries and Main Results}

\subsection{Multivariate approximation in reproducing kernel Hilbert spaces }\

Let $D_d$ be a compact subset of $\mathbb{R}^{d}$, and
$\varrho_{d}$ a probability measure with full support on $D_d$. Let $L_{p}(\varrho_{d})$ $(1\le p\le \infty)$ be the space of
complex-valued measurable  functions with respect to $\varrho_{d}$
with finite  norm
$$\|f\|_{L_{p}(\varrho_{d})}:=\bigg\{\begin{array}{ll}\big(\int_{D_d}|f(\mathbf{x})|^{p}d\varrho_{d}(\mathbf{x})\big)^{1/p},\ \ & 1\le p<\infty, \\
{\rm sup}_{\mathbf{x}\in D}|f(\mathbf{x})|, \
&p=\infty.\end{array}$$We write $L_\infty(D_d)$ and
$\|\cdot\|_{\infty}$ instead of $L_\infty(\varrho_d)$ and
$\|\cdot\|_{L_\infty(\varrho_d)}$. Remark that
$L_{2}(\varrho_{d})$ is a Hilbert space with inner product
$$\langle f,g\rangle_{L_{2}(\varrho_{d})}:=\int_{D_d}f(\mathbf{x})\overline{g(\mathbf{x})}d\varrho_{d}(\mathbf{x}).$$

We will work in the framework of reproducing kernel Hilbert spaces
(see \cite[Chapter 1]{BT} or \cite[Chapter 4]{CS}). Let $H(K_d)$
be a reproducing kernel Hilbert space  with a Hermitian positive
definite continuous kernel $K_d(\mathbf{x},\mathbf{y})$ on
$D_d\times D_d$. Then we have for all ${\bf x} \in D_d$
$$f(\mathbf{x})=\langle f,K_d(\cdot,\mathbf{x})\rangle_{H(K_d)}.$$ It ensures that point evaluations are continuous
functionals on $H(K_d)$. In the sequel we always assume that
\begin{equation}\label{2.1} \|K_d\|_{\infty}:=\sup_{\mathbf{x}\in
D}\sqrt{K_d(\mathbf{x},\mathbf{x})}<\infty,
\end{equation}which implies that $H(K_d)$ is continuously embedded
into $L_{\infty}(D_d)$, i.e.,
\begin{equation}\label{2.2}
\|f\|_{L_{\infty}(D_d)}\leq \|K_d\|_{\infty}\cdot\|f\|_{H(K_d)}.
\end{equation}
Note that we do not need the measure $\varrho_{d}$ for this
embedding.

We consider the multivariate problem $S=\{S_d\}_{d\in \Bbb N}$ in
the  worst case setting, where $S_d: H(K_d)\to G_d$ is a
continuous linear operator from $H(K_d)$ to a Banach space $G_d$
with norm $\|\cdot\|_{G_d}$. We approximate $S_d(f)$ by algorithms
$A_{n,d}(f)$ of the form
\begin{equation}\label{2.3}A_{n,d}(f)=\phi
_{n,d}(L_1(f),L_2(f),\dots,L_n(f)),\end{equation} where
$L_1,L_2,\dots,L_n$ are arbitrary continuous linear functionals or
function values on $H(K_d)$, and
  $\phi _{n,d}:\;\Bbb R^n\to
G_d$ is an arbitrary  mapping. The worst case approximation error
for the algorithm $A_{n,d}$ of the form \eqref{2.3} is defined as
 $$e(S_d; A_{n,d}):=\sup_{ \|f\|_{H(K_d)}\le1}
\|S_d(f)-A_{n,d}(f)\|_{G_d}.$$

The $n$th minimal worst case error
is defined by
$$e(n,S_d;\Lz):=\inf_{A_{n,d} \ {\rm with}\ L_i\in \Lz}e(S_d;A_{n,d}),$$
where $\Lz\in\{\Lz^{\rm all},\Lz^{\rm std}\}$ and  the infimum is
taken over all algorithms of the form \eqref{2.3}. Clearly, we
have
\begin{equation}\label{2.4} e(n,S_d;\Lz^{\rm all})\le e(n,S_d;\Lz^{\rm std}).\end{equation}

For $n=0$, we use $A_{0,d}=0$. We obtain  the so-called initial
error $e(0,S_d; \Lz)$, defined by
$$e(0,S_d):=e(n,S_d;\Lz)=\sup_{\|f\|_{H(K_d)}\le1}\|S_d
(f)\|_{G_d}.$$

Since $H(K_d)$ is a Hilbert space, it follows from \cite[Theorem
4.8]{NW1} that linear algorithms are optimal and hence,
$e(n,S_d;\Lz^{\rm all})$ is equal to the approximation numbers
$a_{n+1}(S_d)$ of $S_d$, defined by
$$a_{n+1}(S_d):=\inf_{\sub{{\rm linear}\ A_{n,d}\,: \, H(K_d)\to G_d  \\
 {\rm rank}\, A_{n,d} \le n}}\sup_{\|f\|_{H(K_d)}\leq 1}\|S_d(f)-A_{n,d}(f)\|_{G_d}.$$
That is,
\begin{equation}\label{2.5}
a_{n+1}(S_d)=e(n,S_d;\Lz^{\rm all}).
\end{equation}
We remark that
$e(n,S_d;\Lz^{\rm std})$ is also called the optimal recovery of
$S_d$ or  the sampling numbers $g_n(S_d)$, i.e.,
\begin{equation}\label{2.6}
g_n(S_d)=e(n,S_d;\Lz^{\rm std}).
\end{equation}

Now for $1\le q\le \infty$, we consider the operators $$I_{q,d}:
H(K_d)\to L_q(\varrho_d)\ \ {\rm with}\ \ I_{q,d}(f)=f .$$
If \eqref{2.1} holds, then $I_{2,d}$ satisfies the finite trace
condition of the kernel
\begin{equation*}
{\rm Tr}\,
(K_d):=\|K_d\|_{2}^{2}=\int_{D_d}K_d(\mathbf{x},\mathbf{x})d\varrho_{d}(\mathbf{x})\le
\|K_d\|_\infty < \infty,
\end{equation*}
 and hence, is compact and Hilbert-Schmidt (see \cite[Lemma 2.3]{SS}).
 From \cite{NW1} we know that
$e(n,I_{2,d};\Lz^{\rm all})$ depends on the eigenpairs
$\big\{(\lz_{k,d},e_{k,d})\big\}_{k=1}^\infty$ of the operator
$$W_d=I_{2,d}^*\,I_{2,d} \colon H(K_d)\to H(K_d)
,$$where  $I_{2,d}^*$ is the adjoint operator of $I_{2,d}$, and
$$\lz_{1,d}\ge \lz_{2,d}\ge \dots \ge \lz_{n,d}\ge \dots \ge0.$$
 That is,   $\{e_{k,d}\}_{k\in \Bbb N}$ is an orthonormal basis in
 $H(K_d)$, and
$$W_d\,e_{k,d}=\lambda_{k,d}\,e_{k,d}.$$

Without loss of generality, we may assume that all the eigenvalues
of $W_d$ are positive.
We set $$\sz_{k,d}=\sqrt{\lz_{k,d}},\ \, \eta_{k,d}=\lz_{k,d}^{-1/2}e_{k,d},\ \ k\in \Bbb N.$$Then
$\sz_{k,d}, \, \eta_{k,d},\ k\in\Bbb N$ are also called   the
singular numbers (values) and singular functions of $I_{2,d}$. By
the Mercer theorem we have
\begin{equation*}
K_d(\mathbf{x},\mathbf{y})= \sum_{k=1}^\infty
e_{k,d}(\mathbf{x})\overline{e_{k,d}(\mathbf{y})}=\sum_{k=1}^\infty\sigma^{2}_{k,d}\eta_{k,d}(\x)\overline{\eta_{k,d}(\y)}.
\end{equation*}

From \cite[p. 118]{NW1}
 we get   that the $n$th minimal worst case
error is
$$e(n,I_{2,d}; \Lz^{\rm all})=a_{n+1}(I_{2,d})=(\lz_{n+1,d})^{1/2}=\sz_{n+1,d},$$
and  it is achieved by the optimal algorithm
$$S_{n,d}^*(f)=\sum_{k=1}^n \langle f, e_{k,d} \rangle_{H(K_d)}\,
e_{k,d},$$that is,
\begin{equation}\label{2.5}e(n,I_{2,d}; \Lz^{\rm all})=\sup_{\|f\|_{H(K_d)}\le 1}\|f-S_{n,d}^*(f)\|_{L_2(\varrho_d)}=(\lz_{n+1,d})^{1/2}.\end{equation}

We remark
that $\{ e_{k,d}\}$ is an orthonormal basis in $H(K_d)$,  $\{
\eta_{k,d}\}$ is an orthonormal system in $L_2(\varrho_d)$, and
for $f\in H(K_d)$,
$$ \langle f, \lz_{k,d}\,\eta_{k,d} \rangle_{{H(K_d)}}= \langle f, \eta_{k,d}
\rangle_{{L_2(\varrho_d)}},$$and
\begin{equation}\label{2.6}S_{n,d}^*(f)=\sum_{k=1}^n \langle f,
\eta_{k,d} \rangle_{L_2(\varrho_d)}\, \eta_{k,d}.\end{equation}

We denote
$$N_{\varrho_{d}}(m,\mathbf{x})=\sum^{m}_{k=1}|\eta_{k}(\mathbf{x})|^{2},\
\ \x\in D_d,$$ and
\begin{equation*}N_{\varrho_{d}}(m)=\|N_{\varrho_{d}}(m,\cdot)\|_\infty.\end{equation*}
The function $(N_{\varrho_{d}}(m,\mathbf{x}))^{-1}$ is often
called Christoffel function in literature (see \cite{NF} and
references therein). The number $(N_{\varrho_{d}}(m))^{1/2}$ is the
exact  constant of the  Nikolskii inequality for $V_m:={\rm
span}\, \{\eta_{1,d},\eta_{2,d},\dots,\eta_{m,d}\}$, i.e.,
$$N_{\varrho_{d}}(m)=\sup_{\sub{f\in V_{m}\\ f\neq 0}} \|f\|_{\infty}^{2}/\|f\|_{L_2(\varrho_d)}^{2}.$$

\subsection{Approximation of multivariate periodic functions}\

Let  $\Bbb T$ denote the torus, i.e., $\Bbb T =[0, 1]$, where the
endpoints of the interval are identified, and $\Bbb T^d=[0,1]^d$
stand for the $d$-dimensional torus. We equip $\Bbb T^d$ with the
 Lebesgue measure $d\x$. Then
$L_{2}(\mathbb{T}^{d})$ is the space of all measurable
$2\pi$-periodic functions $f(\mathbf{x})=f(x_{1},\dots,x_{d})$ on
$\mathbb{T}^{d}$ for which
$$\|f\|_2:=\|f\|_{L_{2}(\mathbb{T}^{d})}:=\Big(\int_{\mathbb{T}^{d}}|f(\mathbf{x})|^{2}\rm d\mathbf{x}\Big)^{1/2} < \infty. $$
Consequently, $\{e^{\rm2\pi i\kk\x} :\ \kk\in \Bbb Z^d\}$ is an
orthonormal basis in $L_2(\Bbb T^d)$, where $\kk\x =\sum_{j=1}^d
k_jx_j$, $\rm i=\sqrt{-1}$. The Fourier coefficients of a function
$f\in L_1(\Bbb T^d)$ are defined as
$$\hat f(\kk)=\int_{\Bbb T^d}f(\x)e^{-\rm 2\pi i\kk\x}\rm d\x,\ \  \kk=(k_1,k_2,\dots, k_d)\in \Bbb
Z^d.$$

Let $\oz$ be a positive function on $\Bbb Z^d$, i.e.,
$\oz(\kk)=\oz(k_1,\dots,k_d)>0$ for all $\kk\in\Bbb Z^d$. We
define the smoothness space $H^\oz(\Bbb T^d)$ by
$$H^\oz(\Bbb T^d)=\Big\{f\in L_2(\Bbb T^d) :
\|f\|_{H^\oz(\Bbb T^d)}=\Big(\sum_{\kk\in\Bbb Z^d}|\hat
f(\kk)|^2\oz(\kk)^2\Big)^{1/2}<\infty\Big\}. $$ Obviously,
$H^\oz(\Bbb T^d)$ is a Hilbert space with inner product
$$\langle f,g\rangle_{{}_{H^\oz(\Bbb T^d)}}=\sum_{\kk\in \Bbb Z^d}\hat
f(\kk)\overline{\hat g(\kk)}\,\oz(\kk)^2 .$$and an orthonormal
basis $\{e_\kk\}_{\kk\in \Bbb Z^d}$,  where
$e_\kk(\x)=\oz(\kk)^{-1}e^{\rm 2\pi i\kk\x}.$

It follows from \cite[Theorem 3.1]{CKS} that $H^\oz(\Bbb T^d)$ is
compactly embedded into $L_\infty(\Bbb T^d)$ or $C(\Bbb T^d)$ if
and only if
$$\sum_{\kk\in\Bbb Z^d}\oz(\kk)^{-2}<\infty.$$
In this case, $H^\oz(\Bbb T^d)$ is a reproducing kernel Hilbert
space with reproducing kernel
\begin{equation}K_d^\oz(\x,\y)=\sum_{\kk\in\Bbb Z^d} \oz(\kk)^{-2} e^{ \rm
2\pi i\kk (\x-\y)},\label{0.0}\end{equation} and
$$\|K_d^\oz\|_\infty=\sum_{\kk\in\Bbb Z^d}\oz(\kk)^{-2}<\infty.$$

We consider the approximation problem
$${\rm APP}_{q,d}: H^\oz(\Bbb T^d)\to L_q(\Bbb T^d), \ \ {\rm
APP}_{q,d}(f)=f, \ 2\le q\le \infty.$$
Note that if $2<q\le \infty$ and
$$\sum_{\kk\in\Bbb Z^d}\oz(\kk)^{-\frac{2q}{q-2}}<\infty,$$then the space
 $H^\oz(\Bbb T^d)$ is compactly embedded into $L_q(\Bbb T^d)$ (see \cite[Proposition 4.12]{CKS}).

It is easily seen that $\big(\oz(\kk)^{-2}, e_\kk\big)_{ \kk\in
\Bbb Z^d}$ are the eigenpairs of the operator $$\widetilde W={\rm
APP}_{2,d}^*\,{\rm APP}_{2,d},$$ where
$e_\kk(\x)=\oz(\kk)^{-1}e^{\rm 2\pi i\kk\x}$.  Let
$\{\lz_{k,d}\}_{k\in\Bbb N}$ be the nonincreasing rearrangement of
the  sequence $\{\oz(\kk)^{-2}\}_{\kk\in \Bbb Z^d}$, and $e_{k,d}$
be the eigenfunction with respect to the eigenvalue $\lz_{k,d}$ of
$\widetilde W$. Then $(\lz_{k,d}, e_{k,d})_{k\in\Bbb N}$ are the
eigenpairs of the operator $\widetilde W$ satisfying$$\lz_{1,d}\ge
\lz_{2,d}\ge \dots\ge \lz_{k,d}\ge \dots>0,\ \ \widetilde W
e_{k,d}=\lz_{k,d}\,e_{k,d},\ k\in\Bbb N.$$ Since
$\{e_{k,d}\}_{k=1}^\infty$ is an orthonormal basis in $H^\oz(\Bbb
T^d)$, we get for any $f\in H^\oz(\Bbb T^d)$,
$$f=\sum_{k=1}^\infty \langle f,e_{k,d}\rangle_{{}_{H^\oz(\Bbb T^d)}}\,
e_{k,d}\ \ \ {\rm and}\ \ \ \|f\|_{H^\oz(\Bbb T^d)}=
\Big(\sum_{k=1}^\infty |\langle f,e_{k,d}\rangle_{{}_{H^\oz(\Bbb
T^d)}}|^2\Big)^{1/2}.
$$

It follow from \cite[Theorem 3.4 and Proposition 4.12]{CKS} that
$$e(n, {\rm APP}_{\infty,d};\Lz^{\rm all})=a_{n+1}({\rm APP}_{\infty,d})=\Big(\sum_{k=n+1}^\infty \lz_{k,d}\Big)^{1/2},$$
and for $2<q<\infty$,
 $$e(n, {\rm APP}_{q,d};\Lz^{\rm all})=a_{n+1}({\rm APP}_{q,d})\le\Big(\sum_{k=n+1}^\infty \lz_{k,d}^{\ \ \frac q{q-2}}\Big)^{\frac{q-2}{2q}}.$$
The initial
error $e(0,{\rm APP}_{\infty,d})$ is given by
$$e(0,{\rm APP}_{\infty,d}):=e(0,{\rm APP}_{\infty,d};\Lz)=\Big(\sum_{k=1}^\infty \lz_{k,d}\Big)^{1/2}.$$

\subsection{Notions of tractability}\

In this paper, we consider the approximation problem ${\rm APP}=
\{{\rm APP}_{\infty,d}\}_{d\in\Bbb N}$. The  information
complexity  can be studied using either the absolute error
criterion (ABS) or the normalized error criterion (NOR). In the
worst case setting for $\star\in\{{\rm ABS,\,NOR}\}$ and  $\Lz\in
\{\Lz^{\rm all}, \Lz^{\rm std}\}$, we define the information
complexity $n^{ \star}(\va ,d;\Lz)$  as
\begin{equation} n^{ \star}(\va
,d;\Lz):=\inf\{n : e(n,{\rm APP}_{\infty,d};\Lz)\le \vz\,
{\rm CRI}_d\},
\end{equation} where
\begin{equation*}
{\rm CRI}_d:=\left\{\begin{split}
 & 1,  &&\text{for $\star$=ABS,} \\
 &e(0,{\rm APP}_{\infty,d}), &&\text{for $\star$=NOR.}
\end{split}\right. \ \ =\ \ \left\{\begin{split}
 &\ 1,  &&\text{for $\star$=ABS,} \\
 &\Big(\sum_{k=1}^\infty \lz_{k,d}\Big)^{1/2}, &&\text{for $\star$=NOR.}
\end{split}\right.
\end{equation*}
Since $\Lz^{\rm std}\subset \Lz^{\rm all},$ we get
\begin{equation*}e(n,d;\Lz^{\rm all})\le e(n,d;\Lz^{\rm std}).\end{equation*}
It follows that for $\star\in\{{\rm ABS,\,NOR}\}$,
\begin{equation} \label{2.11}n^{ \star}(\va
,d;\Lz^{\rm all})\le n^{ \star}(\va ,d;\Lz^{\rm std}).
\end{equation}

In this subsection we recall the various tractability notions in
the worst case setting. First we introduce all notions of
algebraic tractability. Let ${\rm APP}= \{{\rm
APP}_{\infty,d}\}_{d\in\Bbb N}$,
 $\star\in\{{\rm ABS,\,NOR}\}$, and
$\Lz\in \{\Lz^{\rm all}, \Lz^{\rm std}\}$. In the worst case
setting for the class $\Lambda$, and for error criterion $\star$,
we say that ${\rm APP}$ is

$\bullet$ Algebraically strongly polynomially tractable (ALG-SPT)
if there exist $ C>0$ and a non-negative number $p$ such that
\begin{equation}\label{2.12}n^{ \star}(\va ,d;\Lz)\leq C\varepsilon^{-p},\
\text{for all}\ \varepsilon\in(0,1).\end{equation}The exponent
ALG-$p^{ \star}(\Lz)$ of ALG-SPT is defined as the infimum of $p$
for which \eqref{2.12} holds;

$\bullet$ Algebraically  polynomially tractable (ALG-PT)  if there
exist $ C>0$ and non-negative numbers $p,q$ such that
$$n^{ \star}(\va
,d;\Lz)\leq Cd^{q}\varepsilon^{-p},\ \text{for all}\
d\in\mathbb{N},\ \varepsilon\in(0,1);$$

$\bullet$ Algebraically  quasi-polynomially tractable (ALG-QPT) if
there exist $ C>0$ and a non-negative number $t$ such that
\begin{equation}\label{2.13}n^{ \star}(\va
,d;\Lz)\leq C \exp(t(1+\ln{d})(1+\ln{\varepsilon^{-1}})),\
\text{for all}\ d\in\mathbb{N},\
\varepsilon\in(0,1).\end{equation}The exponent ALG-$t^{
\star}(\Lz)$ of ALG-QPT  is defined as the infimum of $t$ for
which \eqref{2.13} holds;

$\bullet$ Algebraically  uniformly weakly tractable (ALG-UWT)  if
$$\lim_{\varepsilon^{-1}+d\rightarrow\infty}\frac{\ln n^{\star}(\va
,d;\Lz)}{\varepsilon^{-\alpha}+d^{\beta}}=0,\ \text{for all}\
\alpha, \beta>0;$$

$\bullet$ Algebraically  weakly tractable (ALG-WT) if
$$\lim_{\varepsilon^{-1}+d\rightarrow\infty}\frac{\ln n^{\star}(\va
,d;\Lz)}{\varepsilon^{-1}+d}=0;$$

$\bullet$ Algebraically  $(s,t)$-weakly tractable (ALG-$(s,t)$-WT)
for fixed $s, t>0$ if
 $$\lim_{\varepsilon^{-1}+d\rightarrow\infty}\frac{\ln n^{\star}(\va
,d;\Lz)}{\varepsilon^{-s}+d^{t}}=0.$$

Clearly, ALG-$(1,1)$-WT is the same as ALG-WT. If ${\rm APP}$ is
not ALG-WT, then ${\rm APP}$  is called  intractable.

If
 the $n$th  minimal error   is exponentially convergent, then we
should study  tractability with $\vz^{-1}$ being replaced by
$(1+\ln {\vz}^{-1})$, which is called exponential tractability.
Recently, there have been many papers studying exponential
tractability (see \cite{CW, DKPW, IKPW, KW, LX2,
X3}).

In the definitions of ALG-SPT, ALG-PT, ALG-QPT, ALG-UWT, ALG-WT,
and ALG-$(s,t)$-WT, if we replace ${\vz}^{-1}$ by $(1+\ln {\vz}^{-1})$, we get the definitions of \emph{exponential strong
polynomial tractability} (EXP-SPT), \emph{exponential polynomial
tractability} (EXP-PT), \emph{exponential quasi-polynomial
tractability} (EXP-QPT), \emph{exponential uniform weak
tractability}
 (EXP-UWT), \emph{exponential weak tractability}
 (EXP-WT), and \emph{exponential $(s,t)$-weak tractability}
 (EXP-$(s,t)$-WT), respectively.  We now give the above notions of exponential tractability in
 detail.

Let ${\rm APP}= \{{\rm APP}_{\infty,d}\}_{d\in\Bbb N}$,
$\star\in\{{\rm ABS,\,NOR}\}$, and $\Lz\in \{\Lz^{\rm all},
\Lz^{\rm std}\}$. In the worst case setting for the class
$\Lambda$, and for error criterion $\star$, we say that ${\rm
APP}$ is

$\bullet$ Exponentially strongly polynomially tractable (EXP-SPT)
if there exist $ C>0$ and a non-negative number $p$ such that
\begin{equation}\label{2.14}n^{ \star}(\va ,d;\Lz)\leq C(\ln\varepsilon^{-1}+1)^{p},\
\text{for all}\ \varepsilon\in(0,1).\end{equation}The exponent
EXP-$p^{\star}(\Lz)$ of EXP-SPT is defined as the infimum of $p$
for which \eqref{2.14} holds;

$\bullet$ Exponentially  polynomially tractable (EXP-PT)  if there
exist $C>0$ and non-negative numbers $p,q$ such that
$$n^{\star}(\va
,d;\Lz)\leq Cd^{q}(\ln\varepsilon^{-1}+1)^{p},\ \text{for all}\
d\in\mathbb{N},\ \varepsilon\in(0,1);$$

$\bullet$ Exponentially  quasi-polynomially tractable (EXP-QPT) if
there exist $C>0$ and a non-negative number $t$ such that
\begin{equation}\label{2.15}n^{\star}(\va
,d;\Lz)\leq C \exp(t(1+\ln{d})(1+\ln(\ln\varepsilon^{-1}+1))),\
\text{for all}\ d\in\mathbb{N},\
\varepsilon\in(0,1).\end{equation}The exponent EXP-$t^{
\star}(\Lz)$ of EXP-QPT  is defined as the infimum of $t$ for
which \eqref{2.15} holds;

$\bullet$ Exponentially  uniformly weakly tractable (EXP-UWT)  if
$$\lim_{\varepsilon^{-1}+d\rightarrow\infty}\frac{\ln n^{\star}(\va
,d;\Lz)}{(1+\ln\varepsilon^{-1})^{\alpha}+d^{\beta}}=0,\ \text{for
all}\ \alpha, \beta>0;$$

$\bullet$ Exponentially  weakly tractable (EXP-WT) if
$$\lim_{\varepsilon^{-1}+d\rightarrow\infty}\frac{\ln n^{\star}(\va
,d;\Lz)}{1+\ln\varepsilon^{-1}+d}=0;$$

$\bullet$ Exponentially  $(s,t)$-weakly tractable (EXP-$(s,t)$-WT)
for fixed $s,t>0$ if
 $$\lim_{\varepsilon^{-1}+d\rightarrow\infty}\frac{\ln n^{\star}(\va
,d;\Lz)}{(1+\ln\varepsilon^{-1})^{s}+d^{t}}=0.$$

\subsection{Main results}\

We shall give  main results of this paper in this subsection.
There are many papers devoted to discussing upper bounds of
$g_n(I_{q,d})\ (1\le q\le \infty)$ in terms of $a_n(I_{2,d})$. The
first upper bound about $g_n(I_{2,d})$ was obtained by Wasilkowski
and Wo\'zniakowski in \cite{WW1} by constructing
 Monte Carlo algorithms.  Using the refined Monte Carlo algorithms, the authors  in \cite{KWW} obtained the upper bounds about $g_n(I_{q,d})\ (1\le q\le \infty)$.
Applying the above upper estimates, the authors in \cite{WW1, KWW,
NW3} obtained some algebraic tractability results for $\Lz^{\rm
std}$ for $I_{q,d}\ (q=2\ {\rm or}\ \infty)$.

If  nodes ${\rm X} = (x^1,\dots, x^n)\in D_d^n$ are drawn
independently and identically distributed according to a
probability measure, then the  samples on the nodes ${\rm X}$ is
called the random information (see \cite{HKNPU, HKNPU2, KS}).
 Krieg and Ullrich
in \cite{KU}  obtained better upper  bounds of $g_n(I_{2,d})$  by
applying random information and weighted least squares algorithms.
Later, the authors in \cite{KUV, KU2, MU, U} extended the results
of \cite{KU}. The authors in \cite{NSU} gave new better  upper
bounds of $g_n(I_{2,d})$ by applying the weighted least squares
method and a new weaver subsampling technique, and finally,  the
authors in \cite{DKU} obtained sharp upper bounds of
$g_n(I_{2,d})$ by using infinite-dimensional variant of
subsampling strategy. The authors in \cite{KSUW} obtained the
power of standard information for exponential tractability for
$L_2$-approximation in the worst case setting. The authors in
\cite{PU} used the weighted least squares method and the
subsampling technique in \cite{NSU} to obtain upper bounds of
$g_n(I_{\infty,d})$ in terms of $a_m(I_{2,d})$ and
$N_{\varrho_d}$.

In this paper we use the weighed least squares method and the
subsampling technique in \cite{DKU}  to get an improved upper
bounds upper bounds of $g_n(I_{\infty,d})$ and $g_n({\rm
APP}_{\infty,d})$. Our result about $g_n({\rm APP}_{\infty,d})$ is
sharp. See the following theorems.

\begin{thm}\label{thm2.1}
 There are absolute constants $c_1,c_2 \in \mathbb{N}$ such that
$$ g_{c_1m}(I_{\infty,d})^2\leq  c_{2} \max \left\{\frac{N_{\varrho_{d}}(m)}{m}\sum_{k\geq \lfloor \frac{m}{2}\rfloor}\sigma^{2}_{k,d} ,\sum_{k\geq \lfloor \frac{m}{4}\rfloor}\frac{N_{\varrho_{d}}(4k)\sigma^{2}_{k,d}}{k}\right\}.$$
\end{thm}

\begin{thm}
There are absolute constants $c_1,c_2 \in \mathbb{N}$ such that
$$ g_{c_1m}({\rm APP}_{\infty,d})\leq c_2 a_{m+1}({\rm APP}_{\infty,d})=c_2\Big(\sum_{k=m+1}^\infty
\lz_{k,d}\Big)^{1/2}.$$In other words,
\begin{equation}e(c_1m,{\rm APP}_{\infty,d};\Lz^{\rm std})\leq c_2 e(m,{\rm APP}_{\infty,d};\Lz^{\rm
all}).\label{2.19-0}\end{equation}
\end{thm}

Based on Theorem 2.2, we obtain two relations between the
information complexities $n^{\star}(\varepsilon,d;\Lambda^{\rm
std})$ and $n^{\star}(\varepsilon,d;\Lambda^{\rm all})$ for
$\star\in\{{\rm ABS,\,NOR}\}$.

\begin{thm} For  $\star\in\{{\rm ABS,\,NOR}\}$, we
have
\begin{equation}\label{2.19}n^{\star}(\varepsilon,d;\Lambda^{\rm std})\le  2{c_{1}} n^{\star}(\frac{\varepsilon}{c_2},d;\Lambda^{\rm
all}),
\end{equation}where $c_1$, $c_2$ are the constants given in
Theorem 2.2.
\end{thm}

In the worst case setting, we study the approximation problem
${\rm APP}= \{{\rm APP}_{\infty,d}\}_{d\in\Bbb N}$. We obtain the
equivalences of various notions of algebraic and exponential
tractability for $\Lz^{\rm all}$ and $\Lz^{\rm std}$ for the
normalized or absolute error criterion without any condition. See
the following theorem.

\begin{thm} Consider the approximation problem ${\rm APP}= \{{\rm APP}_{\infty,d}\}_{d\in\Bbb N}$ for the absolute
or normalized error criterion in the worst  case setting. Then
\vskip 2mm

$\bullet$ $\rm ALG$-$\rm SPT$, $\rm ALG$-$\rm PT$, $\rm ALG$-$\rm
QPT$,   $\rm ALG$-$\rm WT$,  $\rm ALG$-$(s,t)$-$\rm WT$, $\rm
ALG$-$\rm UWT$  for $\Lz^{\rm all}$ is equivalent to $\rm
ALG$-$\rm SPT$, $\rm ALG$-$\rm PT$, $\rm ALG$-$\rm QPT$, $\rm
ALG$-$\rm WT$, $\rm ALG$-$(s,t)$-$\rm WT$, $\rm ALG$-$\rm UWT$ for
$\Lz^{\rm std}$; \vskip 2mm

$\bullet$ $\rm EXP$-$\rm SPT$, $\rm EXP$-$\rm PT$, $\rm EXP$-$\rm
QPT$, $\rm EXP$-$\rm WT$, $\rm EXP$-$(s,t)$-$\rm WT$, $\rm
EXP$-$\rm UWT$  for $\Lz^{\rm all}$ is equivalent to $\rm
EXP$-$\rm SPT$, $\rm EXP$-$\rm PT$, $\rm EXP$-$\rm QPT$,  $\rm
EXP$-$\rm WT$, $\rm EXP$-$(s,t)$-$\rm WT$, $\rm EXP$-$\rm UWT$ for
$\Lz^{\rm std}$; \vskip 2mm

$\bullet$  the exponents of ${\rm SPT}$ are the
same  for  $\Lz^{\rm all}$  and  $\Lz^{\rm std}$, i.e., for $\star
\in \{{\rm ABS, NOR}\}$,
\begin{align*}{\rm ALG}\!-\!p^{\star}(\Lz^{\rm all}) &= {\rm ALG}\!-\!p^{\star}(\Lz^{\rm
std}), \quad  {\rm EXP}\!-\!p^{\star}(\Lz^{\rm all}) = {\rm
EXP}\!-\!p^{\star}(\Lz^{\rm std}).
\end{align*}

\end{thm}

\section{Proofs of Theorems 2.1-2.3}

Theorem 2.1 can be proved in much the same way as \cite[Theorem
2.1]{DKU} and  \cite[Theorem 1]{PU}.  For the convenience of the
reader we give the proof.

Let us keep the notations of Subsection 2.1. We define the
probability density
$$\rho_{m}(\x)=\frac{1}{2}\left(\frac{1}{m} \sum_{k=1}^m\left|\eta_{k,d}(\x)\right|^{2}+\frac{\sum_{k=m+1}^\infty \sigma_{k,d}^{2}\left|\eta_{k,d}(\x)\right|^{2}}
{\sum_{k=m+1}^\infty \sigma_{k,d}^{2}}\right)$$ on $D_d$.  Let
$\x^{1}, \dots, \x^{n} \in D_d $ be  drawn independently and
identically distributed  random points
 according to this density. We
define the infinite-dimensional vectors $y_{1}, \dots, y_{n}$ by
$$\left(y_{i}\right)_{k}=\left\{\begin{array}{cl}\rho_{m}\left(\x^{i}\right)^{-1 / 2} \eta_{k,d}\left(\x^{i}\right) & \text { if } 1 \leq k\leq m , \\ \rho_{m}\left(\x^{i}\right)^{-1 / 2} \gamma_{m}^{-1} \sigma_{k,d} \eta_{k,d}\left(\x^{i}\right) & \text { if } m +1 \leq k<\infty ,\end{array}\right.$$
where
$$\gamma_{m}:=\max \Big\{\sigma_{m+1,d}, \Big(\frac{1}{m} \sum_{k \geq m+1} \sigma_{k,d}^{2}\Big)^{1/2}\Big\}>0. $$

Note that $\rho_{m}(\x^{i}) >0$ almost surely. It follows from
these definitions that $y_{i}\in \ell_{2}$ with
$$\left\|y_{i}\right\|_{2}^{2}=\rho_{m}\left(\x^{i}\right)^{-1}\left(\sum^{m}_{k=1}\left|\eta_{k,d}\left(\x^{i}\right)\right|^{2}+\gamma_{m}^{-2} \sum^{\infty}_{k=m+1} \sigma_{k,d}^{2}\left|\eta_{k,d}\left(\x^{i}\right)\right|^{2}\right) \leq 2 m . $$
and
$$\mathbb{E}\left(y_{i} y_{i}^{*}\right)=\operatorname{diag}\left(1, \ldots, 1, \sigma_{m,d}^{2} / \gamma_{m}^{2}, \sigma_{m+1,d}^{2} / \gamma_{m}^{2}, \ldots\right)=: E .$$
with $\left\|E\right\|_{2\rightarrow 2}=1$ since $\sigma_{m,d}^{2} / \gamma_{m}^{2} \leq 1$ for $k\geq m+1$.
Here, $diag(v)$ denotes a diagonal matrix with diagonal $v$, and  $\left\|\cdot\right\|_{2\rightarrow 2}$ denotes the spectral
norm of a matrix.

 In order to prove Theorem 2.1 we need the following lemmas.

\begin{lem} (See \cite[Theorem
1.1]{MU} and  \cite[Theorem 5.3]{NSU}). Let $ n \geq 3$ and
$y_{1}, \dots, y_{n}$ be $i.i.d.$ random sequences from $\ell_{2}$
satisfying $\left\|y_{i}\right\|_{2}^{2} \leq 2m$ almost surely
and $\left\|E\right\|_{2\rightarrow 2}\leq 1$ with $
E=\mathbb{E}\left(y_{i} y_{i}^{*}\right)$. Then for $0 \leq t \leq
1$
$$\mathbb{P}\left(\left\|\frac{1}{n} \sum_{i=1}^{n} y_{i} y_{i}^{*}-E\right\|_{2 \rightarrow 2}>t\right) \leq 2^{3 / 4} n \exp \left(-\frac{n t^{2}}{42 m}\right) .$$
\end{lem}

Lemma 3.1 gives the  concentration inequality for infinite
matrices. By Lemma 3.1, we know that there exists a deterministic
sample $\x_{1}, \dots, \x_{n} \in D_d $ with $n= \lfloor 10^4
m\log (m+1)\rfloor$ such that the corresponding $y_{1}, \dots,
y_{n} $ satisfy
$$ \quad \left\|\frac{1}{n} \sum_{i=1}^{n} y_{i} y_{i}^{*}-E\right\|_{2 \rightarrow 2} \leq \frac{1}{2} .$$

The following lemma gives an infinite-dimensional version of the
subsampling theorem that might be of independent interest.

\begin{lem} (See \cite[Proposition 13]{DKU}).
There are absolute constants $c_{1}\leq 43200, c_{2}\geq 50, 0< c_{3} <21600 $ with the
following properties. Let $m \in\mathbb{ N},\ n=\lfloor 10^4 m\log
(m+1)\rfloor$, and $y_{1}, \dots, y_{n}$ be vectors from
$\ell_{2}$ satisfying and $\left\|y_{i}\right\|_{2}^{2} \leq 2m$
and
$$\quad\left\|\frac{1}{n} \sum_{i=1}^{n} y_{i} y_{i}^{*}-\left(\begin{array}{cc}I_{m} & 0 \\ 0 & \Lambda\end{array}\right)\right\|_{2 \rightarrow 2} \leq \frac{1}{2}.$$
for some Hermitian matrix $\Lambda$ with
$\left\|\Lambda\right\|_{2\rightarrow 2}\leq1$ where $I_{m}\in
\mathbb{C}^{m\times m} $ denotes the identity matrix. Then, there
is a subset $J\subset \{1,\dots,n\}$ with $\mid J \mid \leq
c_{1}m$, such that
$$c_{2}\left(\begin{array}{cc}I_{m} & 0 \\ 0 & 0\end{array}\right) \leq \frac{1}{m} \sum_{i \in J} y_{i} y_{i}^{*} \leq c_{3} I .$$
We can choose $c_{1} = 43200, c_{2} = 50 $ and $ c_{3} = 21600.$
\end{lem}

\begin{lem}(See \cite[Theorem 2.1]{PU}). Let $$P_m(f):=\sum_{k=1}^n \langle f, \eta_{k,d} \rangle_{L_2(\rho_d)}\,
\eta_{k,d}=\sum_{k=1}^n \langle f, e_{k,d} \rangle_{H(K_d)}\,
e_{k,d}.$$ Then we have
\begin{align*}
 \sup_{\|f\|_{H(K_d)}\leq 1}\|f-P_{m}f\|_{\infty}
 \leq \sqrt{2\sum_{k\geq\lfloor m/4\rfloor}\frac{N_{\varrho_{d}}(4k)}{k}\sigma_{k,d}^{2}}   .\\
\end{align*}
\end{lem}

\noindent{\it Proof of Theorem 2.1. }

Let $f\in H(K_d)$ such that $\|f\|_{H(K_d)}\leq 1$.
According to  lemma 3.1 and lemma 3.2 ,
we obtain points $\x^1,\dots,\x^n \in D_d$
with $n\leq 43200 m$ such that the vectors
$$\left(y_{i}\right)_{k}=\left\{\begin{array}{cl}\rho_{m}\left(\x^{i}\right)^{-1 / 2} \eta_{k,d}\left(\x^{i}\right), & \text { if } 1 \leq k \leq m ,\\ \rho_{m}\left(\x^{i}\right)^{-1 / 2} \gamma_{m}^{-1} \sigma_{k,d} \eta_{k,d}\left(\x^{i}\right), & \text { if } m+1 \leq k<\infty .\end{array}\right.$$
satisfy
$$\Big(\sum_{i=1}^{n} y_{i} y_{i}^{*}\Big)_{\leq m}\geq 50mI ,$$
and
$$\Big(\sum_{i=1}^{n} y_{i} y_{i}^{*}\Big)_{> m}\leq 21600mI ,$$
where we use the notation $A_{\geq m}=(A_{k,l})_{k,l\geq m}$ for
an infinite matrix $A$.

For the above $\mathbf{X}=(\mathbf{x}^{1},\dots,\mathbf{x}^{n})\in
D_d $, we set
\begin{equation*}
{G}:=\left(
\begin{array}{cccc}
\widetilde{\eta}_{1,d}(\x^{1})&\widetilde{\eta}_{2,d}(\x^{1})&\cdots&\widetilde{\eta}_{m,d}(\x^{1})\\
\widetilde{\eta}_{1,d}(\x^{2})&\widetilde{\eta}_{2,d}(\x^{2})&\cdots&\widetilde{\eta}_{m,d}(\x^{2})\\
\vdots&\vdots&\ &\vdots\\
\widetilde{\eta}_{1,d}(\x^{n})&\widetilde{\eta}_{2,d}(\x^{n})&\cdots&\widetilde{\eta}_{m,d}(\x^{n})\\
\end{array}
\right)\in \mathbb{C}^{n\times m},\ \ \
\end{equation*}where $\tilde \eta_{k,d}:=\frac {\eta_{k,d}}
{\sqrt{\rho_{m}}}$. Then  we have the identity
$$G^{\ast}G=\Big(\sum^{n}_{i=1}y_{i}y^{\ast}_{i}\Big)_{\leq m} .$$It
follows that the matrix $G$ has full rank and the spectral norm of
$G^{+}$ is bounded by $(50m)^{-1/2}$, where $G^{+}:=(G^\ast
G)^{-1}G^\ast\in \mathbb{C}^{m\times n}$ be the Moore-Penrose
inverse of the matrix $G$.

We define the weighted least squares estimator
$$A_{n}(f):=
\mathop{\arg\min}_{g\in V_m}
\sum_{i=1}^n\frac{|f(\x^i)-g(\x^i)|^{2}}{\varrho_{m}(\x^i)},$$which
has a unique solution in $V_{m}={\rm
span}\{\eta_{1,d},\eta_{2,d},\dots,\eta_{m,d}\}$.

For all $f\in V_{m}$, we have $A_{n}(f)=f$. Since $G$ is full
rank,  the argmin in the definition of $A_n$ is uniquely defined.
Let $N:H(K_d) \to \mathbb{C}^{n}$ with
$Nf:=(\rho_{m}(\x^{i})^{-1/2}f(\x^{i}))_{1\le i\leq n}$ be the
information mapping. Then the algorithm $A_n$ may be written as
$$A_{n}(f)=\sum_{k=1}^{m}(G^{+}Nf)_{k}\eta_{k,d}.$$

Now we estimate $\|f-A_{n}f\|_{\infty}$. Note that
\begin{align}\|f-A_{n}f\|_{\infty}&\leq \|f-P_{m}f\|_{\infty}+\|P_{m}f-A_n
f\|_{\infty}\notag\\&\le 2\max\Big\{\|f-P_{m}f\|_{\infty},
\|P_{m}f-A_n f\|_{\infty}\Big\}.\label{3.1-0}
\end{align} We have
\begin{align} \|P_{m}f-A_n f\|_{\infty}^2
&=\|A_{n}(f - P_{m}f )\|_{\infty}^2\notag\\
&=\sup_{\mathbf{x}\in D_d}|A_{n}(f - P_{m}f )(\mathbf{x})|^2\notag\\
&=\sup_{\mathbf{x}\in D_d}\left|\sum_{k=1}^{m}(G^{+}N(f-P_{m}f))_{k}\eta_{k,d}(\x)\right|^2\notag\\
&\leq \sup_{\mathbf{x}\in D_d}\sum_{k=1}^{m}|\eta_{k,d}(\mathbf{x})|^{2} \cdot \|G^{+}N(f-P_{m}f)\|_{\ell_{2}^m}^2\notag\\
&\leq {N_{\varrho_{d}}(m)}\cdot \|G^{+}\|_{2 \to
2}^2\cdot\|N(f-P_{m}f)\|_{\ell_{2}^n}^2\notag\\ &\leq \frac 1{50m}
{N_{\varrho_{d}}(m)}
\cdot\|N(f-P_{m}f)\|_{\ell_{2}^n}^2.\label{3.2}
\end{align}
We set $$\Psi:=\big(
\rho_{m}(\x^{i})^{-1/2}\sigma_{k,d}\eta_{k,d}(\x^{i})\big)_{k\geq
m+1,i\leq n}=\big(
\rho_{m}(\x^{i})^{-1/2}e_{k,d}(\x^{i})\big)_{k\geq m+1,i\leq n},
$$and $$ \zeta_{f}:=\big(\langle f,\sigma_{k,d}\eta_{k,d}\rangle_{H(K_d)}\big)_{k\geq m+1}=\big(\langle f,e_{k,d}\rangle_{H(K_d)}\big)_{k\geq m+1},$$
where $\{e_{k,d}\}_{k\ge1}$ is an orthonormal basis in $H(K_d)$.
Obviously, we have$$\|\zeta_{f}\|_{\ell_2}^2=\sum_{k\ge
m+1}|\langle f,e_{k,d}\rangle_{H(K_d)}|^2=\|f-P_mf\|_{H(K_d)}^2\le
1.$$ Thus, we obtain
$${N(f-P_{m}f)}=\Psi\zeta_{f}$$and
$$\Psi^{\ast}\Psi=\gamma_{m}^{2}\Big(\sum_{i=1}^{n}y_{i}y_{\ast}^{\ast}\Big)_{\geq m+1}.$$
It follows from Lemma 3.2 that
$$\|\Psi\|_{2 \to 2}^{2}\leq 21600 m \gamma_{m}^{2}.$$
Hence, we get
\begin{align*}
\|N(f-P_{m}f)\|_{\ell_{2}^n}^{2}&\le\|\Psi\|_{2 \to
2}^{2}\|\zeta_{f}\|_{\ell_{2}}^{2} \leq 21600 m \gamma_{m}^{2}
\\ &\leq 21600 m \max \bigg\{\sigma_{m+1,d}^2, \frac{1}{m} \sum_{k \geq
m+1} \sigma_{k,d}^{2}\bigg\}.
\end{align*}
It follows from \eqref{3.2} that
 \begin{align}
\|P_{m}f-A_n f\|^{2}_{\infty}
&\leq N_{\varrho_{d}}(m) \cdot \frac{1}{50m} 21600m \max \Big\{\sigma_{m+1,d}^2, {\frac{1}{m} \sum_{k \geq m+1} \sigma_{k,d}^{2}}\Big\}\notag\\
&\leq 864\frac{N_{\varrho_{d}}(m)}{m} \sum_{k \geq \lfloor m/2
\rfloor} \sigma_{k,d}^{2},\label{3.3}
\end{align}
where in the last inequality we  use
$$\max \Big\{\sigma_{m+1,d}^2, {\frac{1}{m} \sum_{k \geq m+1} \sigma_{k,d}^{2}}\Big\} \leq \frac{2}{m}\sum_{\lfloor m/2\rfloor}\sigma^{2}_{k,d} .$$

By \eqref{3.1-0}, \eqref{3.3} and  Lemma 3.3, we obtain
\begin{align}g_{c_1m}(I_{\infty,d})^2&\le \sup_{\|f\|_{H(K_d)}\le
1}\|f-A_nf\|\notag\\&\leq c \max
\Big\{\frac{N_{\varrho_{d}}(m)}{m}\sum_{k\geq \lfloor
\frac{m}{2}\rfloor}\sigma^{2}_{k,d} ,\sum_{k\geq \lfloor
\frac{m}{4}\rfloor}\frac{N_{\varrho_{d}}(4k)\sigma^{2}_{k,d}}{k}
\Big\}.\label{3.4}\end{align}

This completes
the proof of Theorem \ref{thm2.1}. $\hfill\Box$

\

\noindent{\it Proof of  Theorem 2.2. }

In the case $I_{\infty,d}={\rm APP}_{\infty,d}$, we have
$N_{\varrho_{d}}(k)=k$. It follows from \eqref{3.4} that
$$ g_{8c_1m}({\rm APP}_{\infty,d})^{2}\leq  c \sum_{k\geq 2m}\sigma^{2}_{k,d}=ca_{2m}({\rm APP}_{\infty,d})^2\le ca_{m+1}({\rm APP}_{\infty,d})^2.$$

Theorem 2.2 is proved. $\hfill\Box$

\

\noindent{\it Proof of Theorem 2.3.}

By  \eqref{2.19-0} we have
$$e(n,{\rm APP}_{\infty,d};\Lz^{\rm std})\leq c_2 e(\lfloor\frac n{c_1}\rfloor,{\rm APP}_{\infty,d};\Lz^{\rm all}).$$
It follows that
\begin{align*}
n^{\star}(\varepsilon,d;\Lambda^{\rm std})&=\min\big\{n\, :\,  e^{\rm }(n,{\rm APP}_{\infty,d};\Lz^{\rm std})\leq\varepsilon{\rm CRI}_d\big\}\nonumber\\
&\leq \min\big\{n : c_2 e^{\rm}(\lfloor\frac{n}{c_1}\rfloor
,{\rm APP}_{\infty,d};\Lz^{\rm
all})\leq\varepsilon{\rm CRI}_d\big\}\nonumber\\
&\le  \min\big\{c_1m+c_1 : e^{\rm}( m,{\rm
APP}_{\infty,d};\Lz^{\rm all})\leq \frac{\varepsilon}{c_2}{\rm
CRI}_d\big\}.
\end{align*}\label{3.1}
Hence, we have
\begin{align*}\label{3.3}n^{\star}(\varepsilon,d;\Lambda^{\rm std})\le c_1+c_1 n^{\star}(\frac{\varepsilon}{c_2},d;\Lambda^{\rm all})\le 2c_1 n^{\star}(\frac{\varepsilon}{c_2},d;\Lambda^{\rm all}) . \end{align*}

Theorem 2.3 is proved.
   $\hfill\Box$

\section{Equivalences  of  tractability for $\Lz^{\rm all}$ and $\Lz^{\rm
std}$ }

 Consider the approximation problem
${\rm APP}= \{{\rm APP}_{\infty,d}\}_{d\in\Bbb N}$ in the worst
case setting for the absolute  or normalized error criterion. Theorem 2.4 gives the equivalences of various notions of algebraic
and exponential tractability for $\Lz^{\rm all}$ and $\Lz^{\rm
std}$. However, the proofs of the equivalences of ALG-tractability
and the ones of EXP-tractability are similar. In this section we
give the proofs of the equivalences of ALG-PT (ALG-SPT), EXP-QPT,
EXP-$(s,t)$-WT, EXP-UWT for $\Lz^{\rm all}$ and $\Lz^{\rm std}$.

\begin{thm}
Consider the problem ${\rm APP}=\{{\rm APP}_{\infty,d}\}_{d\in
\Bbb N}$ in the worst  case   setting for the absolute  or
normalized error criterion. Then,

$\bullet$   ${\rm ALG}$-${\rm PT}$ for $\Lambda^{\rm all}$  is
equivalent to ${\rm ALG}$-${\rm PT}$
 for $\Lambda^{\rm std}$.

$\bullet$  ${\rm ALG}$-${\rm SPT}$  for $\Lambda^{\rm all}$ is
equivalent to ${\rm ALG}$-${\rm SPT}$  for $\Lambda^{\rm std}$. In
this case,  the exponents of ${\rm ALG}$-${\rm SPT}$ for
$\Lambda^{\rm all}$ and  $\Lambda^{\rm std}$ are the same.
\end{thm}

\begin{proof}
It follows from \eqref{2.11} that ALG-PT (ALG-SPT) for $\Lambda^{\rm std}$ means ALG-PT (ALG-SPT) for  $\Lambda^{\rm all}$.
It suffices to show that ALG-PT (ALG-SPT) for $\Lambda^{\rm all}$
 means that ALG-PT (ALG-SPT) for
$\Lambda^{\rm std}$.

Suppose that ALG-PT  holds for $\Lambda^{\rm all}$. Then there
exist $ C\ge 1$ and non-negative $ p,q$ such that
\begin{equation}n^{\star}(\varepsilon,d;\Lambda^{\rm
all})\leq Cd^{q}\varepsilon^{-p},\ \  \text{for all}\ \
d\in\mathbb{N},\ \varepsilon\in(0,1).\label{4.1}\end{equation}
It follows from \eqref{2.19} and \eqref{4.1} that
\begin{align*}
n^{\star}(\varepsilon,d;\Lambda^{\rm std}) &\leq 2c_{1}
n^{\star}(\frac{\varepsilon}{c_2},d;\Lambda^{\rm all}) \leq
2c_{1}Cd^{q}\(\frac{\varepsilon}{c_2}\)^{-p} =: C^{'}
d^{q}\varepsilon^{-p},
\end{align*}which means that ALG-PT  holds for  $\Lambda^{\rm std}$.

If ALG-SPT holds  for $\Lambda^{\rm all}$, then \eqref{4.1}
holds with $q=0$. We obtain
$$ n^\star(\varepsilon,d;\Lambda^{\rm
std})\le C^{'}\varepsilon^{-p},$$which means that
ALG-SPT  holds for  $\Lambda^{\rm std}$. Furthermore, we get
\begin{align*} {\rm ALG\!-\!}p^\star(\Lz^{\rm std})\le
{\rm ALG\!-\!}p^\star(\Lz^{\rm all})\le {\rm ALG\!-\!}p^\star(\Lz^{\rm std}),
\end{align*} which means that the exponents of ${\rm ALG}$-${\rm
SPT}$ for $\Lambda^{\rm all}$ and  $\Lambda^{\rm std}$ are the
same. This completes the proof of Theorem 4.1.
\end{proof}

\begin{thm}
Consider the problem ${\rm APP}=\{{\rm APP}_{\infty, d}\}_{d\in \Bbb N}$ in the worst case setting.
 Then, for the absolute  or normalized  error criterion ${\rm EXP}$-${\rm QPT}$ for $\Lambda^{\rm all}$  is equivalent to ${\rm EXP}$-${\rm QPT}$ for $\Lambda^{\rm std}$.
\end{thm}

\begin{proof}Again,
 it is enough to prove that EXP-QPT for
$\Lambda^{\rm all}$ implies EXP-QPT for
 $\Lambda^{\rm std}$ for the absolute or normalized error criterion.

Suppose that  EXP-QPT holds for $\Lambda^{\rm all}$ for the
absolute or normalized  error criterion. Then  there exist $ C\ge
1$ and non-negative $t$ such that for $\star\in\{{\rm
ABS,\,NOR}\}$,
\begin{equation}\label{5.2}n^{\star}(\va ,d;\Lz^{\rm all})\leq C
\exp(t(1+\ln{d})(1+\ln(\ln\varepsilon^{-1}+1))),\ \text{for all}\
d\in\mathbb{N},\ \varepsilon\in(0,1).\end{equation}
 It follows from
\eqref{2.19} and \eqref{5.2} that
\begin{align}
n^{\star}(\varepsilon,d;\Lambda^{\rm std})
&\leq 2c_{1} n^{\star}(\frac{\varepsilon}{c_2},d;\Lambda^{\rm all})\notag \\
&\leq  2c_{1}C \exp\big(t(1+\ln{d})\big(1+\ln(\ln\varepsilon^{-1}+\ln
c_2 +1))\big) \notag \\
&\le 2c_{1}C\exp\big(t(1+\ln{d})(1+\ln (\ln
c_2+1) +\ln(\ln\varepsilon^{-1}+1))\big)\notag \\
&\le 2c_{1}C \exp\big(
t^*(1+\ln{d})(1+\ln(\ln\varepsilon^{-1}+1))\big),\label{5.2-10}
\end{align}where $t^*=(1+\ln(\ln c_2 +1))t$, and
in the third inequality we use the fact $$\ln (1+a+b)\le
\ln(1+a)+\ln (1+b),\ \ \ a,b\ge 0.$$ The  inequality
\eqref{5.2-10} implies that EXP-QPT holds
for
 $\Lambda^{\rm std}$ for the absolute or normalized  error
 criterion. Theorem 4.2 is proved.
\end{proof}

\begin{rem}From \eqref{2.11} and \eqref{5.2-10} we obtain
$${\rm EXP}\!\!-\!\!t^{\star}(\Lz^{\rm all})\le {\rm EXP}\!\!-\!\!t^{\star}(\Lz^{\rm std})\le (1+\ln(\ln c_2+1)) {\rm EXP}\!\!-\!\!t^{\star}(\Lz^{\rm all}). $$
Similarly, we get
$${\rm ALG}\!\!-\!\!t^{\star}(\Lz^{\rm all})\le {\rm ALG}\!\!-\!\!t^{\star}(\Lz^{\rm std})\le (1+\ln c_2) {\rm ALG}\!\!-\!\!t^{\star}(\Lz^{\rm all}). $$
 Since $c_2>1$, we cannot obtain that
 the exponents $t^{\star}(\Lz^{\rm all})$
and $t^{\star}(\Lz^{\rm std})$ of QPT  are equal.\end{rem}

\begin{thm}
Consider the problem ${\rm APP}=\{{\rm APP}_{\infty, d}\}_{d\in \Bbb N}$  in the worst case setting  for the absolute error criterion
or normalized error criterion. Then for
 fixed $s,t>0$, ${\rm EXP}$-$(s,t)$-${\rm WT}$ for $\Lambda^{\rm all}$  is
equivalent to ${\rm EXP}$-$(s,t)$-${\rm WT}$
 for $\Lambda^{\rm std}$. Specifically, ${\rm EXP}$-${\rm WT}$ for $\Lambda^{\rm all}$  is
equivalent to ${\rm EXP}$-${\rm WT}$
 for $\Lambda^{\rm std}$.
\end{thm}
\begin{proof}
Again,
 it is enough to prove that ${\rm EXP}$-$(s,t)$-${\rm WT}$ for
$\Lambda^{\rm all}$  implies ${\rm EXP}$-$(s,t)$-${\rm WT}$ for
$\Lambda^{\rm std}$.

Suppose that ${\rm EXP}$-$(s,t)$-${\rm WT}$ holds for
$\Lambda^{\rm all}$. Then  we have
\begin{equation}
\label{5.3}\lim_{\varepsilon^{-1}+d\rightarrow\infty}\frac{\ln
n^\star(\varepsilon,d;\Lambda^{\rm
all})}{(1+\ln\varepsilon^{-1})^{s}+d^{t}}=0.
\end{equation}
It follows from \eqref{2.19}  that
\begin{align*}
&\quad\ \frac{\ln n^\star(\varepsilon,d;\Lambda^{\rm
std})}{(1+\ln\varepsilon^{-1})^{s}+d^{t}} \leq\frac{\ln
\Big(2c_{1}n^{\star}(\frac{\varepsilon}{c_2},d;\Lambda^{\rm all})\Big)}{(1+\ln\varepsilon^{-1})^{s}+d^{t}}\\
&\le \frac{\ln
(2c_{1})}{(1+\ln\varepsilon^{-1})^{s}+d^{t}}+\frac{\ln
n^\star(\varepsilon/c_2,d;\Lambda^{\rm
all})}{(1+\ln(\varepsilon/c_2)^{-1})^{s}+d^{t}}\cdot G,
\end{align*}where \begin{align*}G&:=\frac
{(1+\ln(\varepsilon/c_2)^{-1})^{s}+d^{t}}{(1+\ln\varepsilon^{-1})^{s}+d^{t}}\\&\le
\frac {2^s(1+\ln\varepsilon^{-1})^{s}+ 2^s(\ln
c_2)^s+d^{t}}{(1+\ln\varepsilon^{-1})^{s}+d^{t}}\\ &\le 2^s+\frac
{ 2^s(\ln c_2)^s}{(1+\ln\varepsilon^{-1})^{s}+d^{t}}\le
2^s+2^s(\ln c_2)^s.
\end{align*} Since $\varepsilon^{-1}+d\rightarrow\infty$ is
equivalent to $(1+\ln (\varepsilon/c_2)^{-1})^{s}+d^{t}\to
\infty$,  by \eqref{5.3} we get
$$\lim\limits_{\varepsilon^{-1}+d\rightarrow\infty}\frac{\ln
(2c_{1})}{(1+\ln\varepsilon^{-1})^{s}+d^{t}}=0\ \ \ {\rm and} \ \
\lim_{\varepsilon^{-1}+d\rightarrow\infty}\frac{\ln n^\star(\varepsilon/c_2,d;\Lambda^{\rm
all})}{(1+\ln(\varepsilon/c_2)^{-1})^{s}+d^{t}}=0.$$We  obtain
$$\lim\limits_{\varepsilon^{-1}+d\rightarrow\infty} \frac{\ln n^\star(\varepsilon,d;\Lambda^{\rm
std})}{(1+\ln\varepsilon^{-1})^{s}+d^{t}}=0,$$ which implies that
 ${\rm EXP}$-$(s,t)$-${\rm WT}$ holds for
$\Lambda^{\rm std}$.

This completes the proof of
 Theorem 4.4.
\end{proof}

\begin{thm}
Consider the problem ${\rm APP}=\{{\rm APP}_{\infty, d}\}_{d\in \Bbb N}$  in the worst case setting  for the absolute or normalized
error criterion. Then, ${\rm EXP}$-${\rm UWT}$ for $\Lambda^{\rm
all}$ is equivalent to ${\rm EXP}$-${\rm UWT}$
 for $\Lambda^{\rm std}$.
\end{thm}

\begin{proof}
By definition we know  that ${\rm APP}$ is EXP-UWT if and only if
 ${\rm APP}$ is EXP-$(s,t)$-WT for all $s,t>0$. Then Theorem 4.5 follows from Theorem 4.4 immediately.
\end{proof}

\section{Applications of Theorem 2.1 }

This section is devoted to giving the applications of Theorem 2.4
about weighted Korobov spaces and  Korobov spaces with exponential
weight.

First we claim that 
the information complexity of $L_2$
approximation  in the average case setting with the covariance
kernel $K_d^\oz$ given in \eqref{0.0} and the one of ${\rm
APP}=\{{\rm APP}_{\infty,d}\}_{d\in \Bbb N}$ are the same using
$\Lz^{\rm all}$. Consider the approximation problem $\tilde I=\{
\tilde I_d\}_{d\in\Bbb N}$, \begin{equation} \tilde I_d\,\,:
\,\,C([0,1]^d)\to L_2([0,1]^d)\quad \text{with} \quad \tilde
I_d(f)=f.\end{equation}

The space $C([0,1]^d)$ of continuous real functions is equipped
with a zero-mean Gaussian measure $\mu_d$ whose covariance kernel
is given by $K_d^\oz$. We approximate $\tilde I_d \, f$ by
algorithms $A_{n,d}f$ of the form \eqref{2.3} that use $n$
continuous linear functionals $L_i,\ i=1,\dots,n$ on $C([0,1]^d)$.
  The average case error
for $A_{n,d}$ is defined by
\begin{equation*}
e^{\rm avg}(A_{n,d})\;=\;\Big[ \int _{C([0,1]^d)}\big \| \tilde
I_d\,(f)-A_{n,d}(f) \big \|_{L_2([0,1]^d)}^{2}\mu _d(\rm df) \Big
]^{\frac{1}{2}}.
\end{equation*}

The $n$th minimal average case error, for $n\ge 1$, is defined  by
\begin{equation*}
e^{\rm avg}(n,\tilde I_d)=\inf_{A_{n,d}}e(A_{n,d} ),
\end{equation*}
where the infimum is taken over all algorithms of the form
\eqref{2.3}.

 Let $\{\lz_{k,d}\}_{k\in\Bbb N}$ be the
nonincreasing rearrangement of the  sequence
$\{\oz(\kk)^{-2}\}_{\kk\in \Bbb Z^d}$. Then the $n$th minimal
average case error $e^{\rm avg}(n,\tilde I_d)$ is (see \cite{NW1})
$$ e^{\rm avg}(n,\tilde I_d)=\Big(\sum_{k=n+1}^\infty \lz_{k,d}\Big)^{1/2}.$$

For $n=0$, we use $A_{0,d}=0$. We obtain   the so-called initial
error $$e^{\rm avg}(0,\tilde I_d)=e^{\rm
avg}(A_{0,d})=\Big(\sum_{k=1}^\infty \lz_{k,d}\Big)^{1/2} .$$

The information complexity for $\tilde I_d$ in the average case
setting can be studied using either the absolute error criterion
(ABS), or the normalized error criterion (NOR). Then we define the
information complexity $n^{{\rm avg},\star}(\va ,d)$ for $\star\in
\{ {\rm ABS,\, NOR}\}$ as
\begin{equation*}
n^{{\rm avg}, \star}(\va ,\tilde I_d):=\min\{n:\,e^{\rm
avg}(n,\tilde I_d)\leq \va {\rm CRI}_d\},
\end{equation*}where
\begin{equation*}
{\rm CRI}_d=\left\{\begin{matrix}
 & 1, \; \ \ \, \quad\qquad\text{ for $\star$=ABS,} \\
 &e^{\rm avg}(0,\tilde I_d),\quad \text{ for $\star$=NOR.}
\end{matrix}\right.
\end{equation*}
We note that $$e^{\rm avg}(n,\tilde I_d)=e(n, {\rm
APP}_{\infty,d}; \Lz^{\rm all}) \ \ {\rm and}\ \ e^{\rm
avg}(0,\tilde I_d)=e(0, {\rm APP}_{\infty,d}).$$ It follows that
$$n^{{\rm avg}, \star}(\va ,\tilde
I_d)= n^{\star}(\vz, {\rm APP}_{\infty,d}; \Lz^{\rm all}),$$which
shows  the Claim. This means  that using $\Lz^{\rm all}$, ${\rm
ALG}$-tractability and EXP-tractability of various notions for
${\rm APP}=\{{\rm APP}_{\infty, d}\}_{d\in \Bbb N}$ in the worst
case setting and for
 $\tilde I=\{\tilde I_d\}_{d\in\Bbb N}$ in the average case setting are the
 same.

\subsection{Weighted Korobov spaces
$H(K_{d,\mathbf{r},\mathbf{g}})$}\

Let ${\bf r}=\{r_k\}_{k\in \Bbb N} $ and ${\bf g}=\{g_k\}_{k\in
\Bbb N}$ be two sequences  satisfying
\begin{equation}\label{4.11} 1 \geq g_1\geq g_2\geq \cdots\geq g_k \geq \cdots > 0,\end{equation}
and
\begin{equation}\label{4.12} \frac{1}{2}<r_1\le r_2\le \cdots\le r_k\le \cdots. \end{equation}
For $d=1,2,\cdots$, we define the spaces
 $$ H_{d,{\bf r,g}}=H_{1,r_1,g_1}\otimes H_{1,r_2,g_2}\otimes\cdots\otimes H_{1,r_d,g_d}.$$
 Here $H_{1,\az,\beta}$ is the Korobov space of univariate complex valued functions $f$ defined on $[0,1]$ such that
$$\|f\|_{H_{1,\alpha,\beta}}^2:=|\hat f(0)|^2+\beta^{-1}\sum\limits_{h\in \Bbb Z,h\neq0}|h|^{2\az}|\hat f(h)|^2<\infty , $$where $\beta\in(0, 1]$ is a scaling parameter,
and $\alpha>0$ is a smoothness parameter,
$$\hat f(h)=\int_0^1f(x) e^{-2\pi {\rm i} hx}{\rm d}x\ \  {\rm for}\ \ h\in \Bbb
Z$$ are the  Fourier coefficients of $f$, $\rm i=\sqrt{-1}$.  If $\az>
\frac{1}{2}, $ then $H_{1,\az,\beta}$ consists of $1$-periodic
functions  and  is a reproducing kernel Hilbert space with
reproducing kernel
$$R_{\alpha,\beta}(x,y):=1+2\beta\sum\limits_{j=1}^\infty j^{-2\alpha}\cos(2\pi j(x-y)),\  \ x,y\in [0,1].$$
If $\az$ is an
integer, then $H_{1,\az,\beta}$ consists of $1$-periodic functions
$f$ such that $f^{(\az-1)}$ is absolutely continuous,  $f^{(\az)}$
belongs to $L_2([0,1])$, and
$$\|f\|_{H_{1,\alpha,\beta}}^2=\big|\int_{[0,1]}f(x){\rm d}x\big|^2+ (2\pi)^{2\az}\beta^{-1} \int_{[0,1]} |f^{(\az)}(x)|^2{\rm d}x.$$
 See
\cite[Appendix A]{NW1}.

For $d\ge2$ and two sequences ${\bf r}=\{r_k\}_{k\in \Bbb N}$ and ${\bf
g}=\{g_k\}_{k\in \Bbb N}$, the space $H_{d, {\bf \az,\beta}}$ is a Hilbert space
with the inner product $$ \langle f,g\rangle_{H_{d,{\bf r,g}}}=\sum\limits_{{\bf h}\in \Bbb Z^d} \rho_{d,{\bf r,g}}({\bf h})\hat f({\bf h}) \overline{\hat g({\bf h})},$$
where $$\rho_{d,{\bf r,g}}({\bf
h})=\prod\limits_{j=1}^d(\delta_{0,h_j}+g_j^{-1}(1-\delta_{0,h_j}))|h_j|^{2r_j},$$
$\delta_{i,j}=\Big\{\begin{array}{ll}1, \ &i=j,\\ 0, &i\neq
j,\end{array}$\ \ and
$$\hat f{(\bf h)}=\int_{[0,1]^{d}} {f}{(\bf x)}e^{-2\pi {\rm i}{\bf h}{\bf x}}{\rm d}{\bf x}\ \  {\rm for}\ \ {\bf h}\in {\Bbb
Z^{d}},$$
are the  Fourier coefficients of $f$, $\x \cdot
\y=x_1y_1+\dots+x_dy_d$.
If $r_*:=  \mathop{\inf}_{j} r_j> {1}/{2}, $ then $H_{d,{\bf
r,g}}$ consists of $1$-periodic functions on $[0,1]^d$ and  is a
reproducing kernel Hilbert space with  reproducing kernel
\begin{align*}K_{d, {\bf r,g}}(\x,\y)&=\prod_{k=1}^d
R_{r_k,g_k}(x_k,y_k)\\
&=\prod_{k=1}^d\Big(1+2g_k\sum\limits_{j=1}^\infty
j^{-2r_k}\cos(2\pi j(x_k-y_k))\Big),\
 \ \x,\y\in [0,1]^d.\end{align*}
 For integers $r_j$, the inner product of $H_{d, {\bf r,g}}$ can be expressed in terms of
 derivatives, see \cite[Appendix A]{NW1}.

\

We introduce  tractability results of the $L_2$ approximation
problem $\tilde I=\{\tilde I_d\}_{d\in\Bbb N}$, $$\tilde I_d
:\,C([0,1]^d)\to L_{2}([0,1]^d)\ \ {\rm
with}\ \  {\tilde I_d}(f)=f$$
 in the average
case setting for $\Lambda^{\rm all}$.
The space $C([0,1]^d)$ of continuous real functions is equipped
with a zero-mean Gaussian measure $\mu_d$ whose covariance kernel
is given by $K_{d,\mathbf{r},\mathbf{g}}$.

For ALG-tractability of $\tilde I$ for $\Lambda^{\rm all}$, the
sufficient and necessary conditions for ALG-SPT, ALG-PT, ALG-WT
under NOR were given in \cite{LPW1}, and for ALG-SPT, ALG-PT under
ABS in \cite{XX}, for ALG-QPT under NOR in \cite{K, LPW1, X1}, and
for ALG-UWT under ABS or NOR in \cite{X2}, for ALG-$(s, t)$-WT under
NOR or ABS in \cite{CWZ}. We summarize these results as follows.

\begin{thm} \label{thm5.1-0}
Consider the $L_2$ approximation problem $\tilde I=\{\tilde
I_d\}_{d\in\Bbb N}$ in the average case setting with covariance
kernel $K_{d,\mathbf{r},\mathbf{g}}$ and   weights $\{g_k\}_{k\in \Bbb N}$ and
smoothness $\{r_k\}_{k\in \Bbb N}$ satisfying \eqref{4.11} and \eqref{4.12} for
$\Lambda^{\rm all}$.

\rm 1. For ABS or NOR, ALG-SPT holds iff ALG-PT holds iff $$
\liminf\limits_{j \rightarrow \infty} \frac{\ln
\frac{1}{g_{j}}}{\ln j}>1.
$$

\rm 2. For NOR, ALG-QPT holds iff
$$\sup_{d\in \mathbb{N}} \frac{1}{\ln_{+}d}\sum_{j=1}^{d}g_{j}\ln_{+}\frac{1}{g_{j}}<\infty,$$
 where $\ln_{+}x:=\max(1, \ln x)$.

\rm 4. For ABS or NOR,  ALG-UWT holds iff $$
\liminf\limits_{j \rightarrow \infty} \frac{\ln \frac{1}{g_{j}}}{\ln j}\geq1.
$$

\rm 5. For ABS or NOR, ALG-$(s, t)$-WT with $s > 0$ and $t > 1$ always holds.

\rm 6. For ABS or NOR, ALG-$(s, 1)$-WT with $s > 0$ holds iff ALG-WT holds iff $$\lim\limits_{j \rightarrow \infty} g_{j}=0.$$

\rm 7. For ABS or NOR, ALG-$(s, t)$-WT with $s > 0$ and $0<t<1$ holds iff
$$\lim_{j\rightarrow\infty}j^{1-t}g_{j}\ln_{+}\frac{1}{g_{j}}=0.$$

\end{thm}

We consider the $L_\infty$ approximation problem ${\rm APP}=\{{\rm
APP}_{\infty, d}\}_{d\in \Bbb N}$, $${\rm
APP}_{\infty,d}:\,H(K_{d,\mathbf{r},\mathbf{g}})\to
L_{\infty}([0,1]^d)\ \ {\rm with}\ \  {\rm APP}_{\infty,d}(f)=f$$
in the worst case setting.

According to the Claim, ${\rm ALG}$-tractability  of various notions for ${\rm APP}=\{{\rm
APP}_{\infty, d}\}_{d\in \Bbb N}$ in the worst case setting, and
the ones for
 $\tilde I=\{\tilde I_d\}_{d\in\Bbb N}$ in the average case setting are the
 same. By Theorems \ref{thm5.1-0} and Theorem 2.4, we
 obtain the following new results.

\begin{thm}
Consider the $L_\infty$ approximation problem  ${\rm APP}=\{{\rm
APP}_{\infty, d}\}_{d\in \Bbb N}$ defined  over
$H(K_{d,\mathbf{r},\mathbf{g}})$ with weights $\{g_k\}_{k\in \Bbb N}$ and
smoothness $\{r_k\}_{k\in \Bbb N}$ satisfying \eqref{4.11} and \eqref{4.12} in the
worst case setting for $\Lambda^{\rm std}$ and $\Lambda^{\rm
all}$.

\rm 1. For ABS or NOR, ALG-SPT holds iff ALG-PT$$
\liminf\limits_{j \rightarrow \infty} \frac{\ln \frac{1}{g_{j}}}{\ln j}>1.
$$

\rm 2. For NOR, ALG-QPT holds iff
$$\sup_{d\in \mathbb{N}} \frac{1}{\ln_{+}d}\sum_{j=1}^{d}g_{j}\ln_{+}\frac{1}{g_{j}}<\infty.$$

\rm 4. For ABS or NOR,  ALG-UWT holds iff $$
\liminf\limits_{j \rightarrow \infty} \frac{\ln \frac{1}{g_{j}}}{\ln j}>1.
$$

\rm 5. For ABS or NOR, ALG-$(s, t)$-WT with $s > 0$ and $t > 1$ always holds.

\rm 6. For ABS or NOR, ALG-$(s, 1)$-WT with $s > 0$ holds iff ALG-WT holds iff $$\lim\limits_{j \rightarrow \infty} g_{j}=0.$$

\rm 7. For ABS or NOR, ALG-$(s, t)$-WT with $s > 0$ and $0<t<1$ holds iff
$$\lim_{j\rightarrow\infty}j^{1-t}g_{j}\ln_{+}\frac{1}{g_{j}}=0.$$

\end{thm}

\begin{rem}
In \cite{KWW, KWW1}, the authors considered the $L_\infty$
approximation problem ${\rm APP}=\{{\rm APP}_{\infty, d}\}_{d\in
\Bbb N}$  defined  over the weighted Korobov spaces $H(K_d)$ in
the worst case setting for  $\Lz^{\rm all}$ and $\Lz^{\rm std}$
under ABS,  where the reproducing kernel  $K_d$ can be written as
$$K_d(\bf x,\bf y)= \sum_{{\bf h}\in\mathbb{Z}^d}\frac{\cos(2\pi{\bf h}\cdot ({\bf x}-{\bf y}))}{{\bf r}_{\alpha}({\bf \gamma}_{d},\bf h)},$$
 $\alpha>1$ is a smoothness parameter,
$\gamma_{d}=(\gamma_{d,1},\gamma_{d,2},\cdot\cdot\cdot,\gamma_{d,d})$
is a vector of positive weights satisfying $1\geq\gamma_{d,1}\geq
\gamma_{d,2}\geq\cdot\cdot\cdot\geq\gamma_{d,d}>0$, and
$${\bf r}_{\alpha}(\gamma_{d},{\bf h})=\prod_{j=1}^d
{r}_{\alpha}(\gamma_{d,j}, { h}_j)\,\,\ and \, \, {
r}_{\alpha}(\gamma_{d,j}, { h}_j):= \left\{\begin{split}
 & 1,  &&{ h_{j}=0,} \\
 &\gamma^{-1}_{d,j}|h_{j}|^{\alpha}, &&\text{otherwise.}
\end{split}\right.$$

The authors obtained the sufficient and necessary conditions for
ALG-SPT and ALG-PT for the above $L_\infty$ approximation problem.

\end{rem}

\subsection{Korobov spaces
with exponential weight}\

Now we introduce  Korobov kernels with exponential weight. Suppose
that $\mathbf{a}=\{a_{i}\}_{i\in\mathbb{N}}$ and
$\mathbf{b}=\{b_{i}\}_{i\in\mathbb{N}}$ be the positive weights
satisfying
\begin{equation}\label{2.2-0}
0<a_{1}\leq a_{2}\leq \cdots\ \ \ {\rm and}\ \ \
 \b_{\ast}:= \mathop{\inf}_{i\in\mathbb{N}}b_{i}>0.
\end{equation}

For $d=1$, $H(K_{1,\alpha,\beta})$ is a reproducing kernel Hilbert
space with reproducing kernel
   $$ K_{1,\alpha,\beta}(x,y)=\sum_{h\in\mathbb{Z}}\omega^{\alpha |h|^{\beta}}\exp(2\pi ih(x-y)),\ x,y\in[0,1],  \ \ \omega\in(0,1), \ \alpha,\beta>0.$$
Note that if $\beta\ge1$, then all functions in
$H(K_{1,\alpha,\beta})$ are analytic (see \cite{HWW}).

For $d\geq2$,  the  Korobov space $H(K_{d,\mathbf{a},\mathbf{b}})$
with exponential weight consists of complex valued $1$-periodic
continuous functions defined on $[0,1]^{d}$, and is a reproducing
kernel Hilbert space with reproducing kernel
\begin{align*}
K_{d,\mathbf{a},\mathbf{b}}(\mathbf{x},\mathbf{y})&=\prod_{k=1}^d
K_{1,a_k,b_k}(x_k-y_k)\\
&=\sum_{\mathbf{h}\in\mathbb{Z}^{d}}\omega_{\mathbf{h}} \exp(2\pi
{\rm i}\mathbf{h}\cdot(\mathbf{x}-\mathbf{y})),\
\mathbf{x},\mathbf{y}\in[0,1]^{d},\end{align*} where
$\omega_{\mathbf{h}}=\omega^{\sum_{k=1}^{d}a_{k}|h_{k}|^{b_{k}}}$
for all $\mathbf{h}=(h_{1},h_{2},\cdots,h_{d})\in\mathbf{Z}^{d}$
for fixed $\omega\in(0,1)$ and
$\mathbf{h}\cdot(\mathbf{x}-\mathbf{y})=\sum\limits_{k=1}\limits^{d}h_{k}(x_{k}-y_{k}).$

For $f\in H(K_{d,\mathbf{a},\mathbf{b}})$, the norm of $f$ in
$H(K_{d,\mathbf{a},\mathbf{b}})$ is given by
$$\|f\|_{H(K_{d,\mathbf{a},\mathbf{b}})}=(\sum_{\mathbf{h}\in\mathbb{Z}^{d}}\omega_{\mathbf{h}}^{-1}|\hat{f}(\mathbf{h})|^{2})^{\frac{1}{2}},$$where
$$\hat{f}(\mathbf{h})=\int_{[0,1]^{d}}f(\mathbf{x})\exp (2\pi \rm
i\mathbf{h}\cdot\mathbf{x})d\mathbf{x}$$ are the Fourier
coefficients of $f$.

We introduce  previous tractability results. In \cite{KPW}, the
authors considered the approximation problem APP=$\{{\rm
APP}_{\infty,d}\}_{d\in\Bbb N}$,
$${\rm APP}_{\infty,d}:\,H(K_{d,\mathbf{a},\mathbf{b}})\to
L_{\infty}([0,1]^d)\ \ {\rm with}\ \  {\rm APP}_{\infty,d}(f)=f$$
in the worst case setting. They obtained the following results.

\begin{thm}(See \cite[Theorem 1]{KPW}).\ \label{thm5.6}
Consider the $L_\infty$ approximation problem  ${\rm APP}=\{{\rm
APP}_{\infty, d}\}_{d\in \Bbb N}$ defined over  $H(K_{d,{\bf
a},{\bf b}})$  with arbitrary sequences $\mathbf{a}$ and
$\mathbf{b}$ satisfying \eqref{2.2-0} in the worst case setting.
The following results hold for $\Lambda^{\rm all}$ and
$\Lambda^{\rm std}$ under ABS or NOR.

{\rm 1}. ${\rm EXP}$-${\rm  SPT}$ holds iff ${\rm EXP}$-${\rm PT}$
holds iff
$$\sum_{j=1}^{\infty}\frac{1}{b_{j}}<\infty\,\,\,   {\rm and }\,\,\,  \liminf_{j\rightarrow \infty}\frac{\ln a_{j}}{j}>0.$$

{\rm 2}. ${\rm EXP}$-$(s,1)$-${\rm WT}$ for $s\geq1$ holds iff
${\rm EXP}$-${\rm WT}$  holds iff $\lim\limits_{j\rightarrow
\infty} a_{j}=\infty.$

\end{thm}

We introduce  tractability results of the $L_2$ approximation
problem $\tilde I=\{\tilde I_d\}_{d\in\Bbb N}$, $$\tilde I_d
:\,C([0,1]^d)\to L_{2}([0,1]^d),\ \ {\rm
with}\ \  {\tilde I_d}(f)=f$$
 in the average
case setting for $\Lambda^{\rm all}$. The space $C([0,1]^d)$ of
continuous real functions is equipped with a zero-mean Gaussian
measure $\mu_d$ whose covariance kernel is given by
$K_{d,\mathbf{a},\mathbf{b}}$. The ALG-tractability and
EXP-tractability of the $L_2$ approximation problem $\tilde I$ in
the average case setting had been investigated in \cite{CWZ, LX1,
LX2, W2}.

For ALG-tractability of $\tilde I$  for $\Lz^{\rm all}$, the
sufficient and necessary conditions for ALG-SPT, ALG-PT, ALG-UWT,
ALG-WT under NOR or ABS, and for ALG-QPT under NOR were given in
\cite{LX2}, for ALG-$(s, t)$-WT with $s> 0$ and $t\geq1$ under NOR
or ABS in \cite{LX1}, and for $(s, t)$-WT with $s>0$ and $t\in(0,
1)$ under ABS or NOR in \cite{CWZ}. We summarize these results as
follows.

\begin{thm} \label{thm5.7-0}
Consider the $L_2$ approximation  problem $\tilde I=\{\tilde
I_d\}_{d\in\Bbb N}$ in the average case setting with covariance
kernel $K_{d,\mathbf{a},\mathbf{b}}$ and  sequences ${\bf a}$ and
${\bf b}$ satisfying \eqref{2.2-0}  for $\Lambda^{\rm all}$.

\rm 1. For ABS or NOR, ALG-SPT holds iff  ALG-PT holds iff
$$\liminf\limits_{j\rightarrow \infty} \frac{ a_{j}}{\ln j}> \frac{1}{\ln \omega^{-1}}.$$

\rm 2. For NOR, ALG-QPT holds iff
$$\sup_{d\in\mathbb{N}}\frac{1}{\ln_{+}d}\sum_{j=1}^{d}a_{j}\omega^{a_{j}}<\infty.$$

\rm 3. For ABS or NOR, ALG-UWT holds iff
$$\liminf\limits_{j\rightarrow \infty} \frac{ a_{j}}{\ln j}\ge\frac{1}{\ln \omega^{-1}}.$$

\rm 4. For ABS or NOR, ALG-WT holds iff
$$
\lim\limits_{j \rightarrow \infty} {a_{j}}=\infty.
$$

\rm 5. For ABS or NOR, ${\rm ALG}$-$(s,t)$-${\rm WT}$ with $s>0$
and $t>1$ always holds.

\rm 6. For ABS or NOR, ${\rm ALG}$-$(s,1)$-${\rm WT}$ with $s>0$
holds iff  WT
$$
\lim\limits_{j \rightarrow \infty} {a_{j}}=\infty.
$$

\rm 7. For ABS or NOR, ${\rm ALG}$-$(s,t)$-${\rm WT}$ with $s>0$
and $0<t<1$ holds iff
$$
\lim\limits_{j \rightarrow \infty} j^{1-t}a_{j}\omega^{a_{j}}=0.
$$
\end{thm}

For the EXP-tractability of $\tilde I$ under ABS or NOR, the
sufficient and necessary conditions for EXP-SPT, EXP-PT, EXP-UWT,
EXP-WT were given in \cite{LX2}, and for  EXP-$(s, t)$-WT with
$s,t>0$ and $(s,t) \neq (1,1)$ in \cite{W2}. We summarize these
results as follows.

\begin{thm}  \label{thm5.7-1} Consider the $L_2$ approximation  problem $\tilde I=\{\tilde I_d\}_{d\in\Bbb N}$ in the
average case setting with covariance kernel
$K_{d,\mathbf{a},\mathbf{b}}$ and  sequences ${\bf a}$ and ${\bf
b}$ satisfying \eqref{2.2-0}  for $\Lambda^{\rm all}$ under ABS or
NOR.

\rm 1. ${\rm EXP}$-${\rm SPT}$ holds iff ${\rm EXP}$-${\rm PT}$
holds iff
\begin{equation*}
\sum_{j=1}^{\infty}\frac{1}{b_{j}}< \infty\ \ {\rm and }\ \
\liminf\limits_{j\rightarrow \infty} \frac{\ln a_{j}}{j}>0.
\end{equation*}

\rm 2. ${\rm EXP}$-${\rm UWT}$  holds iff
\begin{equation*}
\lim\limits_{j\rightarrow \infty} \frac{\ln a_{j}}{\ln j} =\infty.
\end{equation*}

\rm 3. ${\rm EXP}$-$(s,t)$-${\rm WT}$ with $s>0$ and $t>1$ always
holds.

\rm 4. ${\rm EXP}$-$(s,1)$-${\rm WT}$ with $s\ge 1$ holds iff
EXP-WT holds iff
$$
\lim\limits_{j \rightarrow \infty} {a_{j}}=\infty.
$$

\rm 5. ${\rm EXP}$-$(s,t)$-${\rm WT}$ with $0<s<1$ and $0<t\le1$
holds iff
$$
\lim\limits_{j \rightarrow \infty}
\frac{a_{j}}{j^{(1-s)/s}}=\infty.
$$

\rm 6. ${\rm EXP}$-$(1,t)$-${\rm WT}$ with $t<1$ holds iff
$$
\lim\limits_{j \rightarrow \infty} \frac{{a_{j}}}{\ln j}=\infty.
$$

\rm 7. ${\rm EXP}$-$(s,t)$-${\rm WT}$ with $s>1$ and $t<1$ holds
iff
$$
\lim\limits_{j \rightarrow \infty} j^{1-t}a_{j}\omega^{a_{j}}=0.
$$

\end{thm}

According to the Claim, ${\rm ALG}$-tractability and
EXP-tractability of various notions for ${\rm APP}=\{{\rm
APP}_{\infty, d}\}_{d\in \Bbb N}$ in the worst case setting, and
the ones for
 $\tilde I=\{\tilde I_d\}_{d\in\Bbb N}$ in the average case setting are the
 same.
 By Theorems \ref{thm5.7-0}, \ref{thm5.7-1}, \ref{thm5.6}, and Theorem 2.4, we
 obtain the following new results.
\begin{thm}
Consider the $L_\infty$ approximation problem  ${\rm APP}=\{{\rm
APP}_{\infty, d}\}_{d\in \Bbb N}$ defined over  $H(K_{d,{\bf
a},{\bf b}})$  with  sequences ${\bf a}$ and ${\bf b}$ satisfying
\eqref{2.2-0} in the worst case setting for $\Lambda^{\rm std}$
and $\Lz^{\rm all}$.

\rm 1. For ABS or NOR, ALG-SPT holds iff  ALG-PT holds iff
$$\liminf\limits_{j\rightarrow \infty} \frac{ a_{j}}{\ln j}> \frac{1}{\ln \omega^{-1}}.$$

\rm 2. For NOR, ALG-QPT holds iff
$$\sup_{d\in\mathbb{N}}\frac{1}{\ln_{+}d}\sum_{j=1}^{d}a_{j}\omega^{a_{j}}<\infty.$$

\rm 3. For ABS or NOR, ALG-UWT holds iff
$$\liminf\limits_{j\rightarrow \infty} \frac{ a_{j}}{\ln j}\ge\frac{1}{\ln \omega^{-1}}.$$

\rm 4. For ABS or NOR, ${\rm ALG}$-$(s,t)$-${\rm WT}$ with $s>0$
and $t>1$ always holds.

\rm 5. For ABS or NOR, ${\rm ALG}$-$(s,1)$-${\rm WT}$ with $s>0$
holds iff  ALG-WT holds iff
$$
\lim\limits_{j \rightarrow \infty} {a_{j}}=\infty.
$$

\rm 6. For ABS or NOR, ${\rm ALG}$-$(s,t)$-${\rm WT}$ with $s>0$
and $0<t<1$ holds iff
$$
\lim\limits_{j \rightarrow \infty} j^{1-t}a_{j}\omega^{a_{j}}=0.
$$
\end{thm}

\begin{thm}
Consider the $L_\infty$ approximation problem  ${\rm APP}=\{{\rm
APP}_{\infty, d}\}_{d\in \Bbb N}$ defined over  $H(K_{d,{\bf
a},{\bf b}})$  with  sequences $\mathbf{a}$ and $\mathbf{b}$
satisfying \eqref{2.2-0} in the worst case setting for
$\Lambda^{\rm std}$ and $\Lz^{\rm all}$ under ABS or NOR.

\rm 1. ${\rm EXP}$-${\rm UWT}$  holds iff
\begin{equation*}
\lim\limits_{j\rightarrow \infty} \frac{\ln a_{j}}{\ln j} =\infty.
\end{equation*}

\rm 2. ${\rm EXP}$-$(s,t)$-${\rm WT}$ with $s>0$ and $t>1$ always
holds.

\rm 3. ${\rm EXP}$-$(s,t)$-${\rm WT}$ with $0<s<1$ and $0<t\le1$
holds iff
$$
\lim\limits_{j \rightarrow \infty}
\frac{a_{j}}{j^{(1-s)/s}}=\infty.
$$

\rm 4. ${\rm EXP}$-$(1,t)$-${\rm WT}$ with $t<1$ holds iff
$$
\lim\limits_{j \rightarrow \infty} \frac{{a_{j}}}{\ln j}=\infty.
$$

\rm 5. ${\rm EXP}$-$(s,t)$-${\rm WT}$ with $s>1$ and $t<1$ holds
iff
$$
\lim\limits_{j \rightarrow \infty} j^{1-t}a_{j}\omega^{a_{j}}=0.
$$

\end{thm}

 \noindent{\bf Acknowledgment}  This work was supported by the National Natural Science Foundation of China (Project no.
11671271).

\end{document}